\documentclass[reqno,a4paper]{amsart}
\usepackage[utf8]{inputenc}
\usepackage[T1]{fontenc}
\usepackage{amssymb,amsthm,amsmath}
\usepackage{calrsfs}
\usepackage{hyperref}
\usepackage{mathbbol}
\usepackage{mathabx}
\usepackage[all]{xy}
\usepackage{fancyvrb}
\usepackage{fvextra}
\usepackage{seqsplit}
\usepackage{xcolor}
\usepackage{tikz}
\usetikzlibrary{matrix,arrows}

\newdir{>}{{}*:(1,-.2)@^{>}*:(1,+.2)@_{>}}
\newdir{<}{{}*:(1,+.2)@^{<}*:(1,-.2)@_{<}}
 
\theoremstyle{plain}
\newtheorem{theorem}[subsection]{Theorem}
\newtheorem{proposition}[subsection]{Proposition}
\newtheorem{corollary}[subsection]{Corollary}	
\newtheorem{lemma}[subsection]{Lemma}
\theoremstyle{definition}
\newtheorem{definition}[subsection]{Definition}
\theoremstyle{remark}
\newtheorem{remark}[subsection]{Remark}
\newtheorem{example}[subsection]{Example}

\DeclareMathOperator{\id}{id}
\DeclareMathOperator{\Ker}{Ker}
\DeclareMathOperator{\Coker}{Coker}
\DeclareMathOperator{\Ima}{Im}
\DeclareMathOperator{\Der}{Der}
\DeclareMathOperator{\InnDer}{InnDer}
\DeclareMathOperator{\HH}{H}

\newcommand{\K}{\ensuremath{\mathbb{K}}}

\newcommand{\Lie}{\ensuremath{\mathsf{Lie}}}
\newcommand{\HomLie}{\ensuremath{\mathsf{HomLie}}}
\newcommand{\HomSet}{\ensuremath{\mathsf{HomSet}}}
\newcommand{\Set}{\ensuremath{\mathsf{Set}}}
\newcommand{\rcong}{\mathbin{\rotatebox[origin=c]{-90}{$\cong$}}}

\begin{document}

\title{On the capablility of Hom-Lie algebras}

\author{J.M.~Casas}
\author{X.~García-Martínez}

\email{jmcasas@uvigo.es}
\email{xabier.garcia.martinez@uvigo.gal}

\address[José Manuel Casas]{Universidade de Vigo, Dpto.\ Matemática Aplicada I, E--36005 Pontevedra, Spain}
\address[Xabier García-Martínez]{Universidade de Vigo, Departamento de Matemáticas, Esc.\ Sup.\ de Enx.\ Informática, Campus de Ourense, E--32004 Ourense, Spain\newline
and\newline
Faculty of Engineering, Vrije Universiteit Brussel, Pleinlaan 2, B--1050 Brussel, Belgium}

\thanks{This work was supported by Ministerio de Economía y Competitividad (Spain), with grant number MTM2016-79661-P. The second author is a Postdoctoral Fellow of the Research Foundation–Flanders (FWO)}

\begin{abstract}
A Hom-Lie algebra $(L, \alpha_L)$ is said to be \emph{capable} if there exists a Hom-Lie algebra $(H, \alpha_H)$ such that $L \cong H/Z(H)$. We obtain a characterisation of capable Hom-Lie algebras involving its epicentre and we use this theory to further study the six-term exact sequence in homology and to obtain a Hopf-type formulae of the second homology of perfect Hom-Lie algebras. 
\end{abstract}


\maketitle

%%%%%%%%%%%%%%%%%%%%%%%%%%%%%%%%%%%%%%%%%%%%%%%%%%%%%%%%%%%%
%%%%%%%%%%%%%%%%%%%%%%%%%%%%%%%%%%%%%%%%%%%%%%%%%%%%%%%%%%%%%%%%%%%%%%%%%%%%%%%%%%%%%%%%%%%%%%%%%%%%%%%%%%%%%%%%%%%%%%%%%%%%%%%%%%%%
\section{Introduction}

Hom-Lie algebras were introduced in~\cite{HaLaSi} to study the deformations of the Witt and Virasoro algebras mainly motivated by the study of quantum deformations and discretisation of vector fields via twisted derivations. Their algebraic structure consists in an anticommutative algebra satisfying a twisted version of the Jacobi identity (see Definition~\ref{def}). Since then, many authors extended this idea to many other different algebraic structures, becoming a very prolific research area.

From a categorical-algebraic point of view, the category of Hom-Lie algebras represents a very interesting example that it is worth comprehending better. This may get us closer to understand categorically some complicated algebraic concepts. In fact, the category of Hom-Lie algebras (even adding the property of being multiplicative), it is known to be a semi-abelian category~\cite{JaMaTh} which does not satisfy many of the stronger categorical-algebraic conditions, such as \emph{locally algebraically cartesian closed}~\cite{GrayPhD}, \emph{action representable}~\cite{BoJaKe}, \emph{algebraic coherence}~\cite{acc} or \emph{Normality of Higgins commutators}~\cite{CGrayVdL1}. On the other hand, if we ask the twist to be an automorphism, then it satisfies them all, since it becomes a category of Lie objects over some monoidal category~\cite{GTVV, GaVa2, GaVa, GoVe, GrayRep}. Therefore, Hom-Lie algebras become a \emph{not-so-complicated} example of a \emph{bad-behavioured} semi-abelian algebraic variety. For instance, Hom-Lie algebras were crucial to understand the conditions needed in a semi-abelian category to have a coherent universal central extension theory~\cite{CaVdL}.

In this article, we will study \emph{capable Hom-Lie algebras}, an idea that comes from the concept of capable groups~\cite{Bae}. A group~$G$ is \emph{capable} if there exists a group~$H$ such that~$G \cong H/Z(H)$. In~\cite{BeFeSc}, the \emph{epicentre} was introduced to characterise capable groups: a group is capable if and only if its epicentre is trivial. Later on, a very interesting relation of the epicentre with the non-abelian exterior square was found~\cite{Ell3}.

Capable Lie algebras were introduced in~\cite{SaAlMo} and further studied in~\cite{KhKuLa, NiPaRu}. Nevertheless, the generalisation to the Hom-Lie case is non-trivial due to the loss of some interesting properties, such as the universal central extension condition~\cite{CaVdL}, the difference between the Higgins and Huq commutator, or the fact that it is not known whether its standard homology theory can be obtained from a $Tor$ functor. 

The present manuscript is organised as follows: In Section~\ref{S:prel} we recall several known concepts. Section~\ref{S:tensor} is devoted to introduce the notion of non-abelian exterior product of Hom-Lie algebras, and to find its relation with homology. In Section~\ref{S:exactseq} we obtain a six-term exact sequence involving homology and the non-abelian tensor product, that will be useful in Section~\ref{S:capable} where the capability condition is studied. We define the tensor and exterior centres and relate them with the epicentre of a Hom-Lie algebra. Then, we obtain a characterisation of capable Hom-Lie algebras in terms of their epicentre. Finally, we use the work previously done to further study the six-term exact sequence in homology and to obtain a Hopf-type formulae of the second homology of perfect Hom-Lie algebras.

%%%%%%%%%%%%%%%%%%%%%%%%%%%%%%%%%%%%%%%%%%%%%%%%%%%%%%%%%%%%%%%%%%%%%%%%%%%%%%%%%%%%%%%%%%%%%%%%%%%%%%%%%%%%%%%%%%%%%%%%%%%%%%%%%%%%
\section{Hom-Lie algebras}\label{S:prel}
Throughout this paper we fix $\mathbb{K}$ as a ground field. Vector spaces are considered over $\mathbb{K}$ and linear maps are $\mathbb{K}$-linear maps. We write $\otimes$ (resp.~$\wedge$) for the tensor product $\otimes_\mathbb{K}$ (resp. exterior product $\wedge_\mathbb{K}$ ) over $\mathbb{K}$.

We begin by reviewing some terminology and recalling already known notions used in the paper. We mainly follow~\cite{HaLaSi, JiLi, MaSi2, Yau2}, although with some modifications.

\subsection{Basic definitions}

\begin{definition}\label{def}
	A \emph{Hom-Lie algebra} $(L, \alpha_L)$ is a non-associative algebra $L$ together with a linear
	map $\alpha_L \colon L \to L$ (sometimes called \emph{twist}) satisfying
	\begin{align*}
		&[x,y] = - [y,x], & \text{ (skew-symmetry)}\\
		& \big[\alpha_L(x),[y,z]\big]+\big[\alpha_L(z),[x,y]\big]+\big[\alpha_L(y),[z,x]\big]=0, & \text{(Hom-Jacobi identity)}
	\end{align*}
	for all $x, y, z \in L$.
\end{definition}

In this paper we will only consider the so called \emph{multiplicative Hom-Lie algebras}, i.e., Hom-Lie algebras $(L, \alpha_L)$ such that $\alpha_L$ preserves the product $\alpha_L[x,y]=[\alpha_L(x), \alpha_L(y)]$ for all $x, y \in L$. Nevertheless, as it is standard in the literature, we will omit the word \emph{multiplicative}.	

\begin{example}\label{ejemplo 1} \hfill
	\begin{enumerate}
		\item[a)] Taking $\alpha_L = {\id}_L$, we recover exactly Lie algebras.
		
		\item[b)] Let $V$ be a vector space and $\alpha_V\colon V\to V$ be a linear map, then the pair~$(V, \alpha_V)$ is called \emph{Hom-vector space}. A Hom-vector space $(V,\alpha_V)$ together with the trivial product $[-,-]$ (i.e., $[x,y] = 0$ for any $x,y \in V$) is a Hom-Lie algebra $(V, \alpha_V)$, which is called \emph{abelian Hom-Lie algebra}.
		
		\item[c)] Let $L$ be a Lie algebra, and $\alpha_L\colon L \to L$ be a Lie algebra endomorphism. Then $(L, \alpha_L)$ is a Hom-Lie algebra with the bracket defined by $[x,y]_{\alpha_L} = [\alpha_L(x),\alpha_L(y)]$, for all $x, y \in L$~\cite{Yau2}.
		
		\item[d)] Any Hom-associative algebra $(A, \alpha_A)$ can be endowed with a structure of Hom-Lie algebra by means of the bracket $[a,b]=ab-ba,$ for $a, b \in A$~\cite{MaSi2}.
	\end{enumerate}
\end{example}

Hom-Lie algebras are the objects of the category $\HomLie$, whose morphisms are Lie algebra homomorphisms $f\colon L \to L'$ such that $f \circ \alpha_L = \alpha_{L'} \circ f$. Clearly there is a full embedding ${\Lie} \hookrightarrow{\HomLie}$, $L\mapsto (L,{\id}_L)$, where ${\Lie}$ denotes the category of Lie algebras.

Since $\HomLie$ is a variety of $\Omega$-groups in the sense of Higgins~\cite{Hig2} it is a semi-abelian category, therefore the $3 \times 3$-lemma and the Snake lemma automatically hold~\cite{BoBo}. Below we explicitly study some categorical-algebraic notions in the particular case of the category $\HomLie$ (the general definitions in the semi-abelian context can be found in~\cite{BoJaKe2, BoJaKe, Bou, BoJa3, Jan}).

\begin{definition}
	A \emph{subalgebra} $(H,\alpha_{H})$ of a Hom-Lie algebra $(L, \alpha_L)$ is a vector subspace $H$ of $L$, which is closed under the bracket and invariant under $\alpha_{L}$. In such a case we may write $\alpha_{L\mid}$ for $\alpha_{H}$.
	A subalgebra $(H,\alpha_{L\mid})$ of $(L, \alpha_L)$ is said to be an \emph{ideal} if $[x, y] \in H$ for any $x \in H$, $y\in L$. A Hom-Lie algebra $L$ is called \emph{abelian} if $[x, y] = 0$ for all $x, y \in L$. 
	
	Let $(H,\alpha_{L\mid})$ and $(K,\alpha_{L\mid})$ be ideals of a Hom-Lie algebra $(L,\alpha_L)$. The \emph{(Higgins) commutator} of $(H,\alpha_{L\mid})$ and $(K,\alpha_{L\mid})$, denoted by $([H,K], \alpha_{{L\mid}})$, is the subalgebra of $(L,\alpha_L)$ spanned by the elements $[h,k]$, $h \in H$, $k \in K$. Note that it is not necessarily an ideal, so Huq and Higgins commutators do not always coincide. The idea behind this claim is that the Hom-Jacobi identity may not help to break $\big[x, [h, k]\big]$ into a bracket of elements from $H$ and $K$. A Hom-Lie algebra $(L, \alpha_{L})$ is called \emph{perfect} if $L = [L, L]$. Note that $[L, L]$ is always an ideal. The quotient $\left( \frac{L}{[L,L]}, \overline{\alpha}_L \right)$ is an abelian object in $\HomLie$ and it is called the \emph{abelianisation} of $(L, \alpha_L)$ which we will denote by $\left(L^{\rm ab}, \alpha_{L^{\rm ab}} \right)$. 
\end{definition}

\begin{definition}[\cite{CaGM}]
	The \emph{centre} of a Hom-Lie algebra $(L,\alpha_L)$ is the ideal
	\[
	Z(L) = \{ x \in L \mid [\alpha^k(x), y] =0 \ \text{ for all}\ y \in L, k \in \mathbb{N}\}.
	\]
\end{definition}

\begin{remark}
	When $\alpha_L \colon L \to L$ is a surjective endomorphism, then  we have that $Z(L) = \{ x \in L \mid [x, y] = 0 \}$.
\end{remark}

\begin{definition}[\cite{CaInPa}]\label{alfacentral}
	A short exact sequence of Hom-Lie algebras 
	\[
	0 \to (M, \alpha_M) \stackrel{i} \to (K,\alpha_K) \stackrel{\pi} \to (L, \alpha_L) \to 0
	\]
	is said to be \emph{central} if $[M, K] = 0 $. Equivalently, $M \subseteq Z(K)$.
\end{definition}

Following~\cite{Yau2}, the \emph{homology with trivial coefficients} of a Hom-Lie algebra $(L, \alpha_{L})$ is the homology of the complex $(C_n^{\alpha}(L), d_n), n \geq 1$, where
$C_{n}^{\alpha}\left( L\right) = \Lambda^n L$ and $d_{n}\colon C_{n}^{\alpha}\left( L\right) \longrightarrow C_{n-1}^{\alpha}\left( L\right)$ is given by
\begin{multline*}
	d_{n}\left( x_{1}\wedge\cdots\wedge x_{n}\right) \\ = \underset{1\leqslant i<j\leqslant n}{\sum}\left[ x_{i},x_{j}\right]
	\wedge\alpha_{L}\left( x_{1}\right) \wedge\cdots\wedge\widehat{\alpha_{L}\left( x_{i}\right) }\wedge\cdots\wedge\widehat{\alpha_{L}\left(
		x_{j}\right) }\wedge\cdots\wedge\alpha_{L}\left( x_{n}\right)
\end{multline*}
A routine check shows that $\HH_{0}^{\alpha}(L, \alpha_L) = \K$ and $\HH_{1}^{\alpha}\left( L, \alpha_L \right) =\frac{L}{[L,L]}$. 	

%%%%%%%%%%%%%%%%%%%%%%%%%%%%%%%%%%%%%%%%%%%%%%%%%%%%%%%%%%%%%%%%%%%%%%%%%%%%%%%%%%%%%%%%%%%%%%%%%%%%%%%%%%%%%%%%%%%%%%%%%%%%%%%%%%%%

\subsection{Crossed modules}

\begin{definition}[\cite{CaKhPa}] \label{action}
	Let $(L,\alpha_L)$ and $(M, \alpha_M)$ be Hom-Lie algebras. A \emph{Hom-action} of $(L,\alpha_L)$ on $(M, \alpha_M)$ is a linear map $L \otimes M \to M,$ $x \otimes m\mapsto {}^xm$, satisfying the following properties:
	\begin{enumerate}
		\item[a)] ${}^{[x,y]} \alpha_M(m) = {}^{\alpha_L(x)}({}^y m) - {}^{\alpha_L(y)} ({}^x m)$,
		\item [b)] ${}^{\alpha_L(x)} [m,m'] = [{}^x m, \alpha_M(m')]+[\alpha_M (m), {}^x m']$,
		\item [c)] $\alpha_M({}^x m) = {}^{\alpha_L(x)} \alpha_M(m)$
	\end{enumerate}
	for all $x, y \in L$ and $m, m' \in M$.
	
	A Hom-action is called \emph{trivial} if ${}^xm=0$ for all $x\in L$ and $m\in M$.
\end{definition}

\begin{remark}
	If $(M, \alpha_M)$ is an abelian Hom-Lie algebra enriched with a Hom-action of $(L,\alpha_L)$, then $(M, \alpha_M)$ is a \emph{Hom-module} over $(L,\alpha_L)$~\cite{Yau}.
\end{remark}

\begin{example}\label{ejemplo 2} \hfill
	\begin{enumerate}
		\item[a)] Let $(H, \alpha_{H})$ be a subalgebra and $(K,\alpha_{K})$ an ideal of $(L,\alpha_{L})$. Then there exists a Hom-action of $(H,\alpha_{H})$ on $(K,\alpha_{K})$ given by the product in~$L$. In particular, there is a Hom-action of $(H,\alpha_{H})$ on itself given by the product in $H$.
		
		\item[b)] Let $L$ and $M$ be Lie algebras. Any Lie action of $L$ on $M$ (see e.g.~\cite{Ell1}) defines a Hom-action of $(L, {\id}_{L})$ on $(M, {\id}_{M})$.
		
		\item[c)] Let $L$ be a Lie algebra and $\alpha \colon L \to L$ be an endomorphism. Let $M$ be an $L$-module satisfying the condition ${}^{\alpha(x)}m = {}^x m$, for all $x \in L$, $m \in M$. Then~$(M, {\id}_M)$ is a Hom-module over the Hom-Lie algebra $(L, \alpha)$ considered in Example~\ref{ejemplo 1}~c).
		As an explicit example of this, we can consider $L$ to be the $2$-dimensional vector space with basis $\{ e_1,e_2 \}$, together with the product $[e_1,e_2]=-[e_2,e_1]=e_1$ and zero elsewhere, $\alpha$~to be represented by the matrix $\left( \begin{array}{cc} 1 & 1 \\ 0 & 1 \end{array} \right)$, and $M$ to be the ideal of $L$ generated by $\{e_1\}$.
		
		\item [d)] Any homomorphism of Hom-Lie algebras $f \colon (L,\alpha_L)\to (M,\alpha_M)$ induces a Hom-action of $(L,\alpha_L)$ on $(M,\alpha_M)$ by $^xm = [f(x), m]$, for $x \in L$ and $m \in M$.
		
		\item [e)] Let
		$\xymatrix{ 0 \ar[r] & (M,\alpha_M) \ar[r]^i & (K,\alpha_K) \ar@<-1ex>[r]_{\pi} & (L,\alpha_L) \ar[r] \ar[r]\ar@<-1ex>[l]_{\sigma}& 0}$ be a split extension of Hom-Lie algebras. Then there is a Hom-action of $(L,\alpha_L)$ on $(M,\alpha_M)$ defined in the following way: ${}^{x}m=i^{-1}[\sigma(x),i(m)]$, for all $x\in L$, $m\in M$. In fact, there is an equivalence between split extensions and actions~\cite{CaGM}.
	\end{enumerate}
\end{example}

\begin{definition}[\cite{CaKhPa}]
	Let $(M,\alpha_M)$ and $(N,\alpha_N)$ be Hom-Lie algebras with Hom-actions on each other. The Hom-actions are said to be \emph{compatible} if
	\[
	^{(^mn)} m'=[m',^nm] \quad \text{and} \quad ^{(^nm)}n'=[n',^mn]
	\]
	for all $m,m'\in M$ and $n,n'\in N$.
\end{definition}

\begin{example}
	If $(H,\alpha_H)$ and $(H',\alpha_{H'})$ both are ideals of a Hom-Lie algebra $(L,\alpha_L)$, then the Hom-actions of $(H,\alpha_H)$ and $(H',\alpha_{H'})$ on each other, considered in Example~\ref{ejemplo 2} {a)}, are compatible.
\end{example}

Crossed modules of Hom-Lie algebras were introduced in~\cite{ShCh} in order to prove the existence of a
one-to-one correspondence between strict Hom-Lie 2-algebras and crossed modules of Hom-Lie algebras.

\begin{definition} \label{def cm}
	A \emph{precrossed module} of Hom-Lie algebras is a triple of the form $\big((M, \alpha_M),(L, \alpha_L),\mu\big)$,
	where $(M, \alpha_M)$ and $(L, \alpha_L)$ are Hom-Lie algebras together with a Hom-action from $(L, \alpha_L)$ over $(M, \alpha_M)$ and a Hom-Lie algebra homomorphism $\mu \colon (M, \alpha_M) \to (L, \alpha_L)$ such that the following identity hold:
	\begin{enumerate}
		\item[a)] $\mu(^xm) = [x, \mu(m)]$,
	\end{enumerate}
	for all $m \in M, x \in L$.
	
	A precrossed module $\big((M, \alpha_M),(L, \alpha_L),\mu\big)$ is said to be a \emph{crossed module} when the following identity is satisfied:
	\begin{enumerate}
		\item[b)] $^{\mu(m)} m' = [m,m']$,
	\end{enumerate}
	for $m, m' \in M$.
\end{definition}

\begin{remark} \label{cm properties}
	For a crossed module $\big((M, \alpha_M),(L, \alpha_L),\mu\big)$, the subalgebra $\Ima (\mu)$ is an ideal of $(L,\alpha_L)$ and $\Ker (\mu)$ is contained in the centre of $(M,\alpha_M)$. Moreover $(\Ker (\mu),\alpha_{M \mid})$ is a Hom-$\Coker (\mu)$-module.
\end{remark}

Since $\HomLie$ is a strongly protomodular category (by being a variety of distributive $\Omega_2$-groups~\cite{MaMe}) it satisfies the ``Smith is Huq'' condition~\cite{MaVa}, and therefore this way of introducing crossed modules corresponds to internal crossed modules in the sense of Janelidze~\cite{Jan}.

\begin{example}\label{Ex CM} \hfill
	\begin{enumerate}
		\item[a)] Let $(H,\alpha_H)$ be a Hom-ideal of a Hom-Lie algebra $(L,\alpha_L)$. Then the triple $\big((H,\alpha_H),(L,\alpha_L), inc\big)$ is a crossed module, where the action of $(L,\alpha_L)$ on $(H,\alpha_H)$ is given in Example~\ref{ejemplo 2}~a). There are two particular cases which allow us to think a Hom-Lie algebra as a crossed module, namely $(H,\alpha_H)=(L,\alpha_L)$ and $(H,\alpha_H)=(0,0)$. So $\big((L,\alpha_L),(L,\alpha_L),{\id}\big)$ and $\big((0,0),(L,\alpha_L),0\big)$ are crossed modules.
		\item[b)] Let $(L,\alpha_L)$ be a Hom-Lie algebra and $(M,\alpha_M)$ be a Hom-L-module. Then $\big((M,\alpha_M),(L,\alpha_L),0\big)$ is a crossed module.
	\end{enumerate}
\end{example}

\section{The non-abelian tensor and exterior products of Hom Lie algebras}\label{S:tensor}
\subsection{Non-abelian tensor product of Hom-Lie algebras}

Let us recall the non-abelian tensor product of Hom-Lie algebras introduced in~\cite{CaKhPa} as a generalisation of the non-abelian tensor product of Lie algebras~\cite{Ell2}.

Let $(M,\alpha_M)$ and $(N,\alpha_N)$ be Hom-Lie algebras acting on each other compatibly.
Consider the Hom-vector space $(M\otimes N,\alpha_{M\otimes N})$ given by the tensor product $M\otimes N$ of the underlying vector spaces and the linear map $\alpha_{M \otimes N} \colon M \otimes N \to M \otimes N$, $\alpha_{M \otimes N}(m \otimes n) = \alpha_M(m) \otimes \alpha_N(n)$. Denote by $D(M,N)$ the subspace of $M \otimes N$ generated by all elements of the form
\begin{enumerate}
	\item[a)] $[m,m']\otimes \alpha_N(n) -\alpha_M(m)\otimes {}^{m'}n +\alpha_M(m') \otimes {}^mn$,
	\item[b)] $\alpha_M(m)\otimes [n,n'] - {}^{n'}m\otimes\alpha_N(n)+{}^{n}m\otimes\alpha_N(n')$,
	\item[c)] ${}^{n}m\otimes {}^{m}n$,
	\item[d)]${}^{n}m\otimes {}^{m'}n' +{}^{n'}m'\otimes {}^{m}n$,
	\item[e)]$[{}^{n}m,{}^{n'}m']\otimes \alpha_N({}^{m''}n'')+[{}^{n'}m',{}^{n''}m'']\otimes \alpha_N({}^{m}n)
	+[{}^{n''}m'',{}^{n}m]\otimes \alpha_N({}^{m'}n')$,
\end{enumerate}
for $m,m',m''\in M$ and $n,n',n''\in N$.

The quotient vector space $(M\otimes N)/D(M,N)$ with the product
\begin{equation}\label{pr_in_tensor}
	[m\otimes n, m'\otimes n']=-{}^{n}m\otimes {}^{m'}n'
\end{equation}
and together with the endomorphism $(M\otimes N)/D(M,N)\to (M\otimes N)/D(M,N)$ induced by $\alpha_{M\otimes N}$, is a Hom-Lie algebra, which is called the \emph{non-abelian tensor product of Hom-Lie algebras} $(M,\alpha_M)$ and $(N,\alpha_N)$ (or \emph{Hom-Lie tensor product} for short). It will be denoted by $(M\star N, \alpha_{M\star N})$ and the equivalence class of $m\otimes n$ will be denoted by $m\star n$.

\begin{lemma}[\cite{CaKhPa}]\label{action-on-tensor} 
	Let $(M,\alpha_M)$ and $(N,\alpha_N)$ be Hom-Lie algebras with compatible actions on each other. Then the following statements hold:
	\begin{enumerate}
		\item[a)] There are homomorphisms of Hom-Lie algebras
		\begin{align*}
			\psi_M&\colon (M\star N, \alpha_{M \star N}) \to (M, \alpha_{M }), & \psi_M(m\star n)&= -{}^nm,\\
			\psi_N&\colon (M\star N, \alpha_{M \star N}) \to (N, \alpha_N), & \psi_N(m\star n)&= {}^mn.
		\end{align*}
		\item[b)]
		There are Hom-actions of $(M, \alpha_{M })$ and $(N, \alpha_{N })$ on the Hom-Lie tensor product ($M\star N, \alpha_{M \star N}$) given, for all $m,m'\in M$, $n,n'\in N$, by
		\begin{align*}
			{}^{m'}(m\star n)&=[m',m]\star \alpha_N(n)+\alpha_M(m)\star {}^{m'}n \\
			{}^{n'}(m\star n)&={}^{n'}m\star \alpha_N(n)+\alpha_M(m)\star [n',n]
		\end{align*}
		\item[c)] $\Ker(\psi_M)$ and $\Ker(\psi_N)$ are contained in the centre of $(M\star N, \alpha_{M \star N})$.
		\item[d)] The induced Hom-action of $\Ima(\psi_1)$ on $\Ker(\psi_1)$ and the induced Hom-action action of $\Ima(\psi_2)$ on $\Ker(\psi_2)$ are trivial.
		\item[e)] The homomorphisms $\psi_M$ and $\psi_N$ satisfy the following properties for all $m, m' \in M$, $n, n' \in N$:
		\begin{align*}
			\psi_M(^{m'}(m \star n)) &= [\alpha_M(m'), \psi_M(m \star n)], \\
			\psi_N(^{n'}(m \star n)) &= [\alpha_N(n'), \psi_N(m \star n)], \\
			{^{\psi_M(m \star n)}}(m' \star n') &= [\alpha_{M \star N}(m \star n), m' \star n'] = {^{\psi_N(m \star n)}} (m' \star n').
		\end{align*}
	\end{enumerate}
\end{lemma}

\subsection{Non-abelian exterior product of Hom-Lie algebras}

We introduce now the non-abelian exterior product, following~\cite{DoGaKh, Ell1}.
Let us consider two crossed modules of Hom-Lie algebras $\eta \colon ( M,\alpha_M) \to (L, \alpha_L)$ and $\mu \colon ( N,\alpha_N) \to (L, \alpha_L)$. Then there are induced compatible Hom-actions of $( M,\alpha_M )$ and $(N,\alpha_N )$ on each other via the Hom-action of $(L, \alpha_L)$ (in fact there is an equivalence between pairs of crossed modules with the same domain and compatible actions~\cite{CiMaMe}). Therefore, we can construct the non-abelian tensor product $(M \star N, \alpha_{M \star N})$. We define $(M \Box N, \alpha_{M \Box N})$ as the Hom-vector subspace of $(M \star N, \alpha_{M \star N})$, where $M \Box N$ is the vector subspace spanned by the elements of the form $m \star n, m \in M, n \in N$, such that $\eta(m) = \mu(n)$, and $ \alpha_{M \Box N}$ is the restriction of $\alpha_{M \star N}$ to $M \Box N$.

\begin{proposition}
	The Hom-vector subspace $(M \Box N, \alpha_{M \Box N})$ is contained in the centre of $(M \star N, \alpha_{M \star N})$, so it is an ideal of $(M \star N, \alpha_{M \star N})$.
\end{proposition}
\begin{proof}
	For any $m \star n \in M \Box N, m' \star n' \in M \star N$ we have:
	\begin{align*}
		[\alpha_{M \star N}^k(m \star n), m' \star n'] & = - {^{\alpha_N^k(n)}} \alpha_M^k(m) \star {^{m'}n'}\\
		{}& = - {^{\mu(\alpha_N^k(n))}} \alpha_M^k(m) \star {^{m'}n'}\\
		{}& = - {^{\alpha_L^k( \mu (n) )}} \alpha_M^k(m) \star {^{m'}n'}\\
		{}& = - {^{\alpha_L^k( \eta(m))}} \alpha_M^k(m) \star {^{m'}n'}\\
		{}& = - {^{ \eta (\alpha_M^k(m))}} \alpha_M^k(m) \star {^{m'}n'}\\
		{}& = -[{\alpha_M^k(m)}, \alpha_M^k(m)] \star {^{m'}n'} = 0
	\end{align*}
\end{proof}

\begin{definition}
	Let $\eta \colon ( M,\alpha_M) \to (L, \alpha_L)$ and $\mu \colon ( N,\alpha_N) \to (L, \alpha_L)$ be crossed modules of Hom-Lie algebras. The \emph{non-abelian exterior product} of the Hom-Lie algebras $( M,\alpha_M)$ and $( N,\alpha_N)$ is the quotient 
	\[
	(M,\alpha_M) \curlywedge ( N,\alpha_N) = \left( \frac{M \star N}{M \Box N}, \overline{\alpha}_{M \star N} \right)
	\]
	where $\overline{\alpha}_{M \star N}$ is the induced homomorphism by $\alpha_{M \star N}$ on the quotient.
	
	The coset corresponding to $m \star n$ is denoted by $m \curlywedge n, m \in M, n \in n$.
\end{definition}

\begin{definition}
	Let $\eta \colon ( M,\alpha_M) \to (L, \alpha_L)$ and $\mu \colon ( N,\alpha_N) \to (L, \alpha_L)$ be crossed modules of Hom-Lie algebras. For any Hom-Lie algebra $(P, \alpha_P)$, the bilinear map $h \colon M\times N\to P$ is said to be an \emph{exterior Hom-Lie pairing} if the following properties are satisfied:
	\begin{enumerate}
		\item [a)] $h([m,m'],\alpha_N(n)) = h(\alpha_M(m), {}^{m'} n) - h(\alpha_M(m'), {}^m n)$,
		\item [b)] $h(\alpha_M(m),[n,n']) = h({}^{n'} m, \alpha_N(n)) - h({}^n m, \alpha_N(n'))$,
		\item [c)] $h({}^n m, {}^{m'} n') = - [h(m,n), h(m',n')]$,
		\item [d)] $h(m,n)=0$ whenever $\eta(m) = \mu(n)$,
		\item [e)] $h \circ (\alpha_M \times \alpha_N) = \alpha_P \circ h$,
	\end{enumerate}
	for all $m, m' \in M$, $n, n' \in N$.
	
	An exterior Hom-Lie pairing $h\colon (M \times N, \alpha_{M \times N}) \to (P, \alpha_P)$ is said to be \emph{universal} if for any other exterior Hom-Lie pairing $h'\colon (M \times N, \alpha_{M \times N}) \to (Q, \alpha_Q)$, there is a unique homomorphism of Hom-Lie algebras $\theta \colon (P, \alpha_P) \to (Q, \alpha_Q)$ such that~${\theta \circ h = h'}$.
\end{definition}

\begin{example}\
	\begin{enumerate}
		\item[a)] If $\alpha_M = {\id}_M, \alpha_N = {\id}_N, \alpha_P = {\id}_P,$ and $\alpha_L = {\id}_L$, then the definition of exterior Lie pairing in~\cite{Ell1} is recovered.
		
		\item[b)] Let $( M,\alpha_M), ( N,\alpha_N)$ be Hom-ideals of $(L, \alpha_L)$ and let $\eta, \mu$ be the inclusion maps. Then the bilinear map $h \colon M \times N \to M \cap N, h(m,n) =[m,n],$ is an exterior Hom-Lie pairing.
	\end{enumerate}
\end{example}

\begin{proposition}
	Let $\eta \colon ( M,\alpha_M) \to (L, \alpha_L)$ and $\mu \colon ( N,\alpha_N) \to (L, \alpha_L)$ be crossed modules of Hom-Lie algebras. The map 
	\[
	h\colon (M \times N, \alpha_{M \times N}) \to (M \curlywedge N, \alpha_{M \curlywedge N}), \qquad h(m,n) = m \curlywedge n,
	\] is a universal exterior Hom-Lie pairing.
\end{proposition}

\begin{definition}[\cite{CaKhPa}] \label{alpha identity condition}
	It is said that a Hom-Lie algebra $(L, \alpha_L)$ satisfies the \emph{$\alpha$-identity condition} if 
	\[
	[L, \Ima(\alpha_L- \id_L)]=0
	\]
	which is equivalent to the condition $[x,y]=[\alpha_L(x),y]$ for all $x, y \in L$.
\end{definition}

Any Lie algebra included into $\HomLie$ satisfies the $\alpha$-identity condition and more examples can be found in~\cite{CaKhPa, CaKhPa1}.

\begin{proposition}
	Let $\eta \colon ( M,\alpha_M) \to (L, \alpha_L)$ and $\mu \colon ( N,\alpha_N) \to (L, \alpha_L)$ be crossed modules of Hom-Lie algebras such that $(L, \alpha_L)$ satisfies the $\alpha$-identity condition. Then $\phi \colon (M \curlywedge N, \alpha_{M \curlywedge N}) \to (L, \alpha_L)$, $\phi(m \curlywedge n) = - \eta ({^nm} ) = \mu( {^mn}),$ is a precrossed module.
\end{proposition}
\begin{proof}
	The Hom-Lie action of $(L, \alpha_L)$ on $(M \curlywedge N, \alpha_{M \curlywedge N})$ is given by 
	\[
	{^l(m \curlywedge n)} = {^lm} \curlywedge \alpha_N(n) + \alpha_M(m) \curlywedge {^ln}
	\]
	
	First we need to check that $\phi$ is a homomorphism of Hom-Lie algebras:
	\begin{align*}
		\phi[m \curlywedge n, m' \curlywedge n'] &= \phi(-{^nm} \curlywedge {^{m'}n'}) = \eta \left( {^{(^{m'}n')} ({^nm})} \right) = \eta \left( {^{\mu(^{m'}n')} ({^nm})} \right)\\
		{}&= [\mu(^{m'}n') ,\eta({^nm})] = [\mu(^{\eta(m')}n') ,\eta({^{\mu(n)}m})] \\
		{}&= \big[[\eta(m'),\mu(n')], [\mu(m), \eta(n)]\big] = - [\eta(-{^nm}), \eta(- ^{n'}m')]\\
		{}&= [\phi(m \curlywedge n), \phi(m' \curlywedge n')]
	\end{align*}
	
	Obviously $\alpha_L \circ \phi = \phi \circ \alpha_{M \curlywedge N}$.
	
	Then, we check that it is satisfies the precrossed module condition:
	\begin{align*}
		\phi\big({^l(m \curlywedge n)}\big) &= \phi({^lm} \curlywedge \alpha_N(n) + \alpha_M(m) \curlywedge {^ln})\\
		{}&= - \eta \left( {^{\alpha_N(n)}{(^lm)}} \right) - \eta \left( {^{(^ln)}\alpha_M(m)} \right) \\
		{}&= - \eta \left( {^{\mu(\alpha_N(n))}{(^lm)}} \right) - \eta \left( {^{\mu(^ln)}\alpha_M(m)} \right)\\
		{}&= -[\alpha_L(\mu(n)), \eta{(^lm)}] - \eta \left({^{[l,\mu(n)]}}\alpha_M(m) \right) \\
		{}&= -\big[\alpha_L(\mu(n)), [l,\eta{(m)}]\big] -\big[[l,\mu(n)], \alpha_L(\eta(m))\big]\\
		{}&= - \big[\alpha_L(l),[\mu(n),\eta(m)]\big]\\
		{}&= - \big[l,[\mu(n),\eta(m)]\big] = [l, \eta(- {^nm})] = [l, \phi(m \curlywedge n)]
	\end{align*}
	
\end{proof}

There is a surjective homomorphism of Hom-Lie algebras $\pi \colon ( M,\alpha_M) \star ( N,\alpha_N) \to ( M,\alpha_M) \curlywedge( N,\alpha_N)$ given by $\pi(m \star n) = m \curlywedge n$.

Let $(M,\alpha_M), (N, \alpha_N)$ be ideals of a Hom-Lie algebra $(L, \alpha_L)$. According to Example~\ref{Ex CM}~a), they can be seen as crossed modules through the inclusion in~$(L, \alpha_L)$. Hence
\begin{equation} \label{exterior ideal}
	( M,\alpha_M) \curlywedge ( N,\alpha_N) = \left( \frac{M \star N}{\{m \star m \mid m \in M \cap N \}}, \overline{\alpha}_{M \star N} \right)
\end{equation}

\begin{proposition}
	Let $(M,\alpha_M)$ and $(N, \alpha_N)$ be ideals of a Hom-Lie algebra~$(L, \alpha_L)$. There is a homomorphism of Hom-Lie algebras 
	\[
	\theta_{M,N} \colon ( M,\alpha_M) \curlywedge ( N,\alpha_N) \to ( M,\alpha_M) \cap ( N,\alpha_N)
	\]
	given by $\theta_{M,N} (m \curlywedge n) = [m,n]$, for all $m \in M, n \in N$. Moreover $\theta_{M,N}$ is a crossed module of Hom-Lie algebras when that $(L, \alpha_L)$ satisfies the $\alpha$-identity condition.
\end{proposition}
\begin{proof}
	The Hom-action of $x \in ( M,\alpha_M) \cap ( N,\alpha_N)$ over $m \curlywedge n \in ( M,\alpha_M) \curlywedge ( N,\alpha_N)$ is given by
	\begin{align*}
		{}^x(m \curlywedge n) &= [x,m] \curlywedge \alpha_N(n) + \alpha_M(m) \curlywedge [x, n],\\
		\theta_{M,N}(^x(m \curlywedge n)) &= \big[[x,m],\alpha_N(n)\big] + \big[\alpha_M(m),[x,n]\big] = - \big[\alpha_{L \mid}(x),[n,m]\big]\\
		{}&= \big[x,[m,n]\big] = {^x\theta_{M,N}(m \curlywedge n)},\\
		{}^{\theta_{M,N}(m \curlywedge n)}(m' \curlywedge n') &= \big[[m,n],m'\big] \curlywedge \alpha_N(n')+\alpha_M(m') \curlywedge \big[[m,n],m'\big] \\
		{}&= \alpha_{L \mid} [m, n] \curlywedge \alpha_{L \mid} [m', n'] = [m,n] \curlywedge [m',n'] \\
		{}&= [m \curlywedge n, m' \curlywedge n'].
	\end{align*}
\end{proof}

\begin{proposition} \label{uce exterior}
	Let $(L, \alpha_L)$ be a perfect Hom-Lie algebra. Then,
	\[
	(L, \alpha_L) \star (L, \alpha_L) = (L, \alpha_L) \curlywedge (L, \alpha_L)
	\]
	and the homomorphism $\theta_{L, L} \colon (L, \alpha_L) \curlywedge (L, \alpha_L) \to (L, \alpha_L)$ is the universal central extension of $(L, \alpha_L)$. Moreover $\Ker(\theta_{L,L}) \cong \HH_2^{\alpha}(L, \alpha_L)$. 
\end{proposition}
\begin{proof}
	By Equation~\eqref{pr_in_tensor}, when $(L, \alpha_L)$ is perfect the ideal $L \Box L$ is zero and therefore $(L \curlywedge L, \alpha_{L \curlywedge L}) = (L \star L, \alpha_{L \star L})$. Then, the second part is an immediate consequence of Theorem~4.3 and Theorem~4.4 in~\cite{CaKhPa}.
\end{proof}

\begin{lemma} \label{1}
	If $(N, \alpha_N)$ is an ideal of a Hom-Lie algebra $(L, \alpha_L)$, then the following induced sequence of Hom-Lie algebras 
	\[
	\xymatrix{
		(N \curlywedge L, \alpha_{N \curlywedge L}) \ar[r] & (L \curlywedge L, \alpha_{L \curlywedge L}) \ar[r]^-{\pi \curlywedge \pi} & (\frac{L}{N} \curlywedge \frac{L}{N}, \overline{\alpha}_{L \curlywedge L}) \ar[r] & 0
	}
	\]
	is exact.
\end{lemma}

\begin{proof}
	It is a special case of the exact sequence of~\cite[Proposition~3.12]{CaKhPa}. After taking quotients, we can use the relation $x \wedge y = - y \wedge x$, which holds in $L \wedge L$ to erase one of the factors of the semi-direct product.
\end{proof}

Free Hom-Lie algebras were constructed in~\cite{CaGM}, where the following adjoint functors were obtained:
\[
\xymatrix{ \HomSet \ar@<1ex>[r]^{\mathcal{F}_r} \ar@{}[r]|-{\bot} & \ar@<1ex>[l]^{\mathcal{U}} \HomLie}
\]
where $\HomSet$ is the category of sets with a chosen endomorphism, $\mathcal{F}_r$ is the functor that assigns to a Hom-set $(X, \alpha_X)$ the free Hom-Lie algebra $\mathcal{F}_r(X, \alpha_X)$ and $\mathcal{U}$ is the functor that assigns to a Hom-Lie algebra $(B,\alpha_B)$ the Hom-set obtained by forgetting the operations. Note that since $\HomLie$ is a variety, it is monadic over $\Set$ and therefore it has enough projectives.

On the other hand, since $\HomLie$ is semi-abelian the quotient
\begin{equation} \label{multiplier}
	\frac{(R, \alpha_R) \cap [(F,\alpha_F), (F, \alpha_F)]}{[(F, \alpha_F), (R, \alpha_R)]}
\end{equation}
is a Baer invariant~\cite[Theorem 6.9]{EvVdL}, i.e., it doesn't depend on the chosen free presentation $0 \to (R, \alpha_R) \to (F, \alpha_F) \overset{\rho} \to (G, \alpha_G) \to 0$. This is the \emph{Schur multiplier} of the Hom-Lie algebra $(G, \alpha_G)$ and we denote it as $\mathcal{M}(G, \alpha_G)$.

\begin{lemma}[\cite{CaVdL}] \label{ext perfect}
	Let $(L, \alpha_L)$ be a Hom-Lie algebra. For any central extension
	\[
	0 \to (N, \alpha_N) \to (G, \alpha_G) \overset{\pi} \to (L, \alpha_L) \to 0
	\]
	the extension 
	\[
	0 \to (N, \alpha_N) \cap [(G, \alpha_G), (G, \alpha_G)] \to [(G, \alpha_G), (G, \alpha_G)] \to ([L, L], \alpha_{L\mid}) \to 0
	\]
	is also central. Moreover, if $(L, \alpha_L)$ is perfect the commutator $[(G, \alpha_G), (G, \alpha_G)]$ is also perfect.
\end{lemma}

Let $(L, \alpha_L)$ be a Hom-Lie algebra and let
\[
0 \to (S, \alpha_S) \to (F, \alpha_F) \overset{\tau} \to (L, \alpha_L) \to 0
\]
be a free presentation. Then $([F,S], \alpha_{F \mid})$ is an ideal of $(F, \alpha_F)$, so there exists a surjective homomorphism 
\[
\overline{\tau} \colon \frac{(F, \alpha_F)}{([F,S], \alpha_{F \mid})} \twoheadrightarrow (L, \alpha_L),\qquad \overline{\tau}(f+[F,S]) = \tau(f).
\]
Moreover, %$\Ker(\overline{\tau}) = \dfrac{(S, \alpha_S)}{([F,S], \alpha_{F \mid})}$ and
\begin{equation} \label{central extension}
	0 \to \left( \frac{S}{[F, S]}, \widetilde{\alpha}_S \right) \to \left( \frac{F}{[F, S]}, \widetilde{\alpha}_F \right) \overset{\overline{\tau}}\to (L, \alpha_L) \to 0
\end{equation}
is a central extension. By Lemma~\ref{ext perfect}, the following extension
\begin{equation*}
	0 \to \frac{(S, \alpha_S)}{([F,S], \alpha_{F \mid})} \bigcap \left( \frac{[F, F]}{[F,S]}, \alpha_{F \mid}\right) \to \left(\frac{[F, F]}{[F, S]}, \widetilde{\alpha}_{F\mid} \right)  \overset{{\widetilde{\overline{\tau}}}}\to ([L,L],\alpha_{L \mid}) \to 0
\end{equation*} 
is also central. If in particular $(L, \alpha_L)$ is perfect, the commutator $\left( \frac{[F, F]}{[F,S]}, \alpha_{F \mid}\right)$ is also perfect and the central extension is universal. Indeed, for any other central extension $0 \to (A, \alpha_A) \to (K, \alpha_K) \overset{\sigma} \to (L, \alpha_L) \to 0$ of $(L, \alpha_L)$, there is a homomorphism of Hom-Lie algebras such that the following diagram commutes
\[ \xymatrix{
	0 \ar[r] & (S, \alpha_S) \ar[r] & (F, \alpha_F) \ar[r]^-{\tau} \ar@{-->}[d]^-{h} & (L, \alpha_L) \ar[r] \ar@{=}[d] & 0\\
	0 \ar[r] & (A, \alpha_A) \ar[r] & (K, \alpha_K) \ar[r]^-{\sigma} & (L, \alpha_L) \ar[r] & 0
} \]
In fact, $h$ induces a homomorphism $\overline{h} \colon \left( \frac{F}{[F,S]}, \overline{\alpha}_F \right) \to (K, \alpha_K)$ such that $\sigma \circ \overline{h} = \overline{\tau}$. The restriction of $\overline{h}$ to $\left( \frac{[F, F]}{[F,S]}, \alpha_{F \mid}\right)$ provides the required homomorphism, which is unique due to~\cite[Lemma 4.7]{CaInPa}.

Since the universal central extension of a perfect Hom-Lie algebra is unique up to isomorphism, then Proposition~\ref{uce exterior} implies that $(L \curlywedge L, \alpha_{L \curlywedge L}) \cong \left( \frac{[F, F]}{[F,S]}, \alpha_{F \mid}\right)$ and $\HH_2^{\alpha}(L, \alpha_L) \cong \left( \frac{S \cap [F, F]}{[F, S]}, \alpha_{F \mid} \right)$ for any perfect Hom-Lie algebra $(L, \alpha_L)$.

\begin{theorem}
	Let $(L, \alpha_L)$ be a perfect Hom-Lie algebra and let 
	\[
	0 \to (S, \alpha_S) \to (F, \alpha_F) \overset{\tau} \to (L, \alpha_L) \to 0
	\]
	be a free presentation of $(L, \alpha_L)$. Then,
	\[
	\HH_2^{\alpha}(L, \alpha_L) \cong \left( \frac{S \cap [F, F]}{[F, S]}, \alpha_{F \mid} \right)
	\]
\end{theorem}

\begin{lemma}\label{12}
	If $F$ is a free Hom-Lie algebra, then $F \curlywedge F \cong [F, F]$.
\end{lemma}

\begin{proof}
	It is done following the same strategy as in~\cite{Ell2}.
\end{proof}

%%%%%%%%%%%%%%%%%%%%%%%%%%%%%%%%%%%%%%%%%%%%%%%%%%%%%%%%%%%%%%%%%%%%%%%%%%%%%%%%%%%%%%%%%%%%%%%%%%%%%%%%%%%%%%%%%%%%%%%%%%%%%%%%%%%%
\section{Stallings-Stammbach exact sequence}\label{S:exactseq}

Let $0 \to (N, \alpha_N) \to (G, \alpha_G) \overset{\pi} \to (L, \alpha_L) \to 0$ be a short exact sequence of Hom-Lie algebras, and let $0 \to (R, \alpha_R) \to (F, \alpha_F) \overset{\rho} \to (G, \alpha_G) \to 0$ be a projective presentation of $(G, \alpha_G)$, we can construct the following diagram of projective presentations:

\begin{equation} \label{free present diagr}
	\begin{gathered}
		\xymatrix{& & 0 \ar[d] & 0 \ar[ld]\\
			& & (R, \alpha_R) \ar[d]\ar[ld] \\
			0 \ar[r] & (S, \alpha_S) \ \ar[r] \ar[d] & (F, \alpha_F) \ar[d]^-{\rho} \ar[rd]^-{\tau = \pi \circ \rho} \\
			0 \ar[r]& (N, \alpha_N) \ar[r] \ar[d]& (G, \alpha_G) \ar[r]^-{\pi} \ar[d]& (L, \alpha_L) \ar[r] \ar[dr] & 0 \\
			& 0 & 0 & & 0
		}
	\end{gathered}
\end{equation}

Based on the $3 \times 3$-lemma, we can write the commutative diagram of Figure~\ref{F:diag} to obtain the following result:

\begin{figure}
	\hspace*{-1.4cm}
	\Small{
		\xymatrix @C=0.001cm{
			& ([F,R], \alpha_{R \mid}) \ar@{>->}[d] \ar@{>->}[rrd]& & & & \\
			& ([F,S] \cap R, \alpha_{R \mid}) \ar@{>>}[d] \ar@{>->}[rr] \ar@{>->}[ddl]& & ([F,F] \cap R, \alpha_{R \mid}) \ar@{>>}[d] \ar@{>>}[drr] \ar@{>->}[ddl]& & \\
			& \left( \frac{[F,S] \cap R}{[F,R]}, \alpha_{R \mid} \right) \ar@{>->}[rr] \ar@{>->}[dd]& & { \left( \frac{[F,F] \cap R}{[F,R]}, \alpha_{R \mid} \right)} \ar@{>>}[rr] \ar@{>->}[dd]& & \left( \frac{[F,F] \cap R}{[F,S] \cap R}, \alpha_{R \mid} \right) \ar@{>->}[dd] \ar@{>->}[dl]\\
			([F,S], \alpha_{S \mid}) \ar@{>->}[rr] \ar@{>>}[dd]& & ([F,F] \cap S, \alpha_{S \mid}) \ar@{>>}[rr] \ar@{>>}[ldddd]& & { \left( \frac{[F,F] \cap S}{[F,S]}, \widetilde{\alpha}_{S \mid} \right)} \ar@{>>}[llddddd] & \\
			& \bullet \ar@{>>}[ddd] \ar@{>->}[rr]& & \left( \frac{[F,F]}{[F,R]}, \widetilde{\alpha}_{F \mid} \right) \ar@{>>}[ddd] \ar@{>>}[rr] & & \bullet \ar@{>>}[dd] \\
			\left( [N,G], \alpha_{N \mid} \right) \ar@{}[d]|{\rcong}  & & & & & \\
			\left( \frac{[F,S]}{[F,S] \cap R}, \widetilde{\alpha}_{R \mid} \right) \ar@{>->}[rd] \ar@{>->}[rddd] 	& & & & & ([Q, Q], \alpha_{Q \mid}) \ar@{}[d]|{\rcong} \\
			& (N \cap [G,G], \alpha_{N \mid}) \ar@{>->}[rr] \ar@{>->}[dd] \ar@{>>}[rd]& & ([G,G], \alpha_{G \mid}) \ar@{>>}[rr] \ar@{>->}[dd] & &  \cong \left( \frac{[G,G]}{N \cap [G,G]}, \widetilde{\alpha}_{G \mid} \right) \ar@{>->}[dd] \\
			& & \left(\frac{N \cap [G,G]}{[G,N]}, \alpha_{N \mid} \right) \ar@{>->}[dd]& & & \\
			& (N, \alpha_N) \ar@{>->}[rr] \ar@{>>}[dd] \ar@{>>}[rd]& & (G, \alpha_G) \ar@{>>}[rr] \ar@{>>}[dd]& & (Q, \alpha_Q) \ar@{>>}[dd] \\
			& & { \left( \frac{N}{[G,N]},\widetilde{\alpha}_N \right)} \ar@{>>}[ld]& & & \\
			& \left( \frac{N}{N \cap [G, G]},\alpha_{N \mid} \right) \ar@{>->}[rr] & & { \left( \frac{G}{[G,G]}, \alpha_{G \mid} \right)} \ar@{>>}[rr] & & { \left(\frac{Q}{[Q,Q]}, \alpha_{Q \mid}\right)}
	} }
	\caption{}\label{F:diag}
\end{figure}

\begin{theorem}
	Let $0 \to (N, \alpha_N) \to (G, \alpha_G) \overset{\pi} \to (L, \alpha_L) \to 0$ be a short exact sequence of Hom-Lie algebras. There exists the following natural exact sequence:
	\begin{equation}\label{5-term}
		\mathcal{M}(G, \alpha_G) \to \mathcal{M}(L, \alpha_L) \to \left( \frac{N}{[G,N]},\widetilde{\alpha}_N \right) \to \HH^{\alpha}_1(G, \alpha_G) \to \HH^{\alpha}_1(L, \alpha_L) \to 0
	\end{equation}
\end{theorem}

\begin{remark}
	If $\alpha_G = {\id}_G$ and $\alpha_L = {\id}_L$, sequence~\eqref{5-term} is the Stallings-Stammbach exact sequence associated to a short exact sequence of Lie algebras~\cite{HiSt}.
\end{remark}

If $0 \to (N, \alpha_N) \to (G, \alpha_G) \overset{\pi} \to (L, \alpha_L) \to 0$ is a central extension of Hom-Lie algebras, then sequence~~\eqref{5-term} gives rise to the following natural exact sequence
\begin{equation*}% \label{5-term central}
	\mathcal{M}(G, \alpha_G) \to \mathcal{M}(L, \alpha_L) \to \left( N,\alpha_N \right) \to \HH^{\alpha}_1(G, \alpha_G) \to \HH^{\alpha}_1(L, \alpha_L) \to 0.
\end{equation*}
Moreover the kernel of $\mathcal{M}(G, \alpha_G) \to \mathcal{M}(L, \alpha_L)$ is $\left( \frac{[F,S]}{[F,R]}, \widetilde{\alpha}_{S \mid} \right)$, since $[F, S] \subseteq R$ under the centrality condition.

Since $(N, \alpha_N)$ and $\left( \frac{G}{[G,G]}, \alpha_{G \mid} \right)$ are abelian Hom-Lie algebras, we can construct the Hom-vector space $\left( N \otimes \frac{G}{[G,G]}, \alpha_{\otimes} \right)$, where $\alpha_{\otimes}(n \otimes \overline{g}) = \alpha_N(n) \otimes \overline{\alpha}_G(\overline{g})$. Therefore there is a well-defined surjective homomorphism 
\[
\varphi \colon \left(N \otimes \frac{G}{[G,G]}, \alpha_{\otimes} \right) \to \left( \frac{[F,S]}{[F,R]}, \widetilde{\alpha}_{S \mid} \right), \qquad \varphi(n \otimes \overline{g}) = [f,s] + [F,R],
\]
where $\rho(s) =n, \rho(f) =g, s \in S, f \in F$. The composition 
\[
G\colon \left( N \otimes \frac{G}{[G,G]}, \alpha_{\otimes} \right) \overset{\varphi} \twoheadrightarrow \left( \frac{[F,S]}{[F,R]}, \widetilde{\alpha}_{S \mid} \right) \hookrightarrow \mathcal{M}(G, \alpha_G)
\]
gives rise to the following six-term exact sequence:
\begin{equation}\label{6-term central}
	\begin{tikzpicture}[baseline=(current  bounding  box.center)]
		\matrix(m) [matrix of nodes,row sep=1em, column sep=2em, text height=2.8ex, text depth=1.5ex]
		{
			$\bigg(N \otimes \frac{G}{[G,G]}, \alpha_{\otimes}\bigg)$  & $\mathcal{M}(G, \alpha_G)$  & $\mathcal{M}(L, \alpha_L)$ & \mbox{}  \\
			\mbox{}	& $(N,\alpha_N)$  & $\HH^{\alpha}_1(G, \alpha_G)$ & $\HH^{\alpha}_1(L, \alpha_L)$ & 0 \\};
		\path[overlay,->, font=\scriptsize,>=angle 90]	
		(m-1-1) edge (m-1-2)
		(m-1-2) edge (m-1-3)
		(m-1-3) edge [out=355,in=175] (m-2-2)
		(m-2-2) edge  (m-2-3)
		(m-2-3) edge (m-2-4)
		(m-2-4) edge (m-2-5)
		;
	\end{tikzpicture}
\end{equation}

%%%%%%%%%%%%%%%%%%%%%%%%%%%%%%%%%%%%%%%%%%%%%%%%%%%%%%%%%%%%%%%%%%%%%%%%%%%%%%%%%%%%%%%%%%%%%%%%%%%%%%%%%%%%%%%%%%%%%%%%%%%%%%%%%%%%

\section{Capable Hom-Lie algebras}\label{S:capable}
\begin{definition}
	A Hom-Lie algebra $(L, \alpha_L)$ is said to be \emph{capable} if there exists a Hom-Lie algebra $(H, \alpha_H)$ such that $L \cong H/Z(H)$.
\end{definition}

When $\alpha_L = {\id}_L$ and $\alpha_H = {\id}_H$ the above definition recovers the notion of capable Lie algebra in~\cite{SaAlMo}. It is well-known that capability of groups (respectively, Lie algebras) is closely related with the group of the inner automorphisms (respectively, the Lie algebra of the inner derivations). Let us recall some notions concerning derivations from~\cite{Sheng}. We denote by $\alpha^k$ the composition of $\alpha$ with itself $k$ times.

\begin{definition}
	An \emph{$\alpha^k$-derivation} of a Hom-Lie algebra $(L, \alpha_L)$ is a linear map $d \colon L \to L$ such that
	\begin{enumerate}
		\item[a)] $ d \circ \alpha_L = \alpha_L \circ d$,
		\item[b)] $d[x,y] = [d(x), \alpha^k(y)] + [\alpha^k(x), d(y)]$, for all $x, y \in L$.
	\end{enumerate}
\end{definition}

We denote by $\Der_{\alpha^k}(L)$ the set of all $\alpha^k$-derivations of~$(L, \alpha_L)$. 
The algebra 
\[
\Der(L) = \bigoplus_{k \geq 0} {\Der}_{\alpha^k}(L)
\] 
is a Hom-Lie algebra with respect to the usual bracket operation $[d, d'] = d \circ d' - d' \circ d$ and the endomorphism $\widetilde{\alpha} \colon {\Der}(L) \to {\Der}(L)$ given by $\widetilde{\alpha}(d) = \alpha \circ d$.

For any Hom-Lie algebra $(L, \alpha_L)$ satisfying the $\alpha$-identity condition (Definition~\ref{alpha identity condition}) and $x \in L$, we define $d_k(x) \colon L \to L$ by $d_k(x) (y) = [\alpha^k(x), y]$. 
Note that in this setting it holds that $d_k(x)(y) = [x, y]$, for any $k \in \mathbb{N}$, and therefore $d_k(x) = d_h(x) = [x, -]$ for any $k, h \in \mathbb{N}$.
Then $d_k(x) \in \Der_{\alpha^{k+1}}(L)$, which is called an inner $\alpha^{k+1}$-derivation. We denote by $\InnDer_{\alpha^k}(L)$ the set of all inner $\alpha^k$-derivations, and 
\[
\InnDer(L) = \bigoplus_{k \geq 0} \InnDer_{\alpha^k}(L)
\] 
is an ideal of ${\Der}(L)$. 

There is a homomorphism of Hom-Lie algebras 
\[
\varphi \colon L \to {\Der}(L), \quad \varphi(x) = (d_0(x), d_1(x), \dots, d_k(x), \dots )
\] 
such that $\Ima(\varphi) = \InnDer(L)$ and $\Ker(\varphi) = Z(L)$. 
This homomorphism shows that if a Hom-Lie algebra $(L, \alpha_L)$ satisfying the $\alpha$-identity condition (Definition~\ref{alpha identity condition}) is isomorphic to inner derivations of some Hom-Lie algebra $(H, \alpha_H)$ that satisfies the $\alpha$-identity condition, then $(L, \alpha_L)$ is capable, i.e.,  we can obtain the following exact sequence 
%This homomorphism shows that a Hom-Lie algebra $(L, \alpha_L)$ satisfying the $\alpha$-identity condition (Definition~\ref{alpha identity condition}) is capable if and only if it is isomorphic to inner derivations of some Hom-Lie algebra $(H, \alpha_H)$ that satisfies the $\alpha$-central identity condition, i.e., we can obtain the following exact sequence 
\[
0 \to (Z(H), \alpha_{H \mid}) \to (H, \alpha_H) \overset{\varphi}\to (\InnDer(H), \widetilde{\alpha}_H) \cong (L, \alpha_L) \to 0.
\]

\begin{definition}
	Let $(L, \alpha_L)$ be a Hom-Lie algebra.
	The \emph{tensor centre} of $(L, \alpha_L)$ is the set:
	\[
	Z_{\alpha}^{\star}(L) = \{ l \in L \mid \alpha^k(l) \star x =0, ~ {\rm for ~ all} ~x \in L, k \in \mathbb{N} \}.
	\]
	The \emph{exterior centre} of $(L, \alpha_L)$ is the set:
	\[
	Z_{\alpha}^{\curlywedge}(L) = \{ l \in L \mid \alpha^k(l) \curlywedge x =0, ~ {\rm for ~ all} ~x \in L, k \in \mathbb{N} \}.
	\]
\end{definition}

\begin{lemma}
	For any Hom-Lie algebra $(L, \alpha_L)$ both $Z_{\alpha}^{\star}(L)$ and $Z_{\alpha}^{\curlywedge}(L)$ are ideals of $(L, \alpha_L)$ contained in $Z(L)$.
\end{lemma}
\begin{proof}
	Let $l \in Z_{\alpha}^{\curlywedge}(L)$ and $x \in L$. We need to prove that $[l, x] \in Z_{\alpha}^{\curlywedge}(L)$, i.e., $\alpha_L^{k}([l, x]) \curlywedge y = 0$ for all $k \in \mathbb{N}$ and $y \in L$. But this is true since we already have that $\alpha_L^{k}([l, x]) = [\alpha^k_L(l), \alpha^k_L(x)] = \theta_{L, L}\big(\alpha^k_L(l) \curlywedge \alpha^k_L(x)\big) = 0$. The same argument works for $Z_{\alpha}^{\star}(L)$ with the tensor adapted version of $\theta_{L, L}$.
\end{proof}

It is obvious that $Z_{\alpha}^{\star}(L) \subseteq Z_{\alpha}^{\curlywedge}(L)$. The equality follows whenever $(L, \alpha_L)$ is perfect by Proposition~\ref{uce exterior}.

\begin{proposition} \label{equivalence}
	Let $0 \to (S, \alpha_S) \to (F, \alpha_F) \overset{\tau} \to (L, \alpha_L) \to 0$ be a free presentation of $(L, \alpha_L)$ and let $\psi \colon (C, \alpha_C) \to (L, \alpha_L)$ be the central extension~\eqref{central extension} $\left( {\rm i.e.} ~(C, \alpha_C) = \left( \frac{F}{[F,S]}, \alpha_{F \mid}\right) \right)$. Then, there is an isomorphism of Hom-Lie algebras $\frac{[F, F]}{[F, S]} = [C, C] \cong L \curlywedge L$. Moreover, $x \in Z(C)$ if and only if $\psi(x) \in Z_{\alpha}^{\curlywedge}(L)$.
\end{proposition}

\begin{proof}
	The first part is a consequence of Lemmas~\ref{1} and~\ref{12}, where the isomorphism $\left( L \curlywedge L, \alpha_{L \curlywedge L} \right) \cong \left( \frac{[F, F]}{[F,S]}, \alpha_{F \mid}\right)$ is induced by the map $l_1 \curlywedge l_2 \mapsto [x, y]+[F,S]$, such that $\tau(x)=l_1, \tau(y) = l_2$.
	
	Let $x \in Z(C)$. Then,
	\[
	0 = [\alpha^k(x),y] \equiv l_1 \curlywedge l_2 = \psi \left( \alpha^k(x) \right) \curlywedge \psi(y) = \alpha^k\left( \psi(x)\right) \curlywedge \psi(y)
	\]
	hence $\psi(x) \in Z_{\alpha}^{\curlywedge}(L)$.
	
	Conversely, let $\psi(x) \in Z_{\alpha}^{\curlywedge}(L)$. Since 
	\[
	[\alpha^k(x), y] \equiv \alpha^k(l_1) \curlywedge l_2 = \alpha^k\big(\psi(x)\big) \curlywedge \psi(y) = 0,
	\]
	then $x \in Z(C)$.
\end{proof}

Proposition~\ref{equivalence} implies that $\psi\big(Z(C)\big) \subseteq Z_{\alpha}^\curlywedge(L)$. On the other hand, $\psi^{-1}(x) \subseteq Z(C)$ for any $x \in Z_{\alpha}^\curlywedge(L)$, therefore
\begin{equation} \label{identity}
	\psi\big(Z(C)\big) = Z_{\alpha}^\curlywedge(L).
\end{equation}

\begin{theorem}
	A Hom-Lie algebra $(L, \alpha_L)$ is capable if and only if $Z_{\alpha}^\curlywedge(L)=0$.
\end{theorem}
\begin{proof}
	Assume $Z_{\alpha}^\curlywedge(L)=0$. Let $\psi \colon (C, \alpha_C) \twoheadrightarrow (L, \alpha_L)$ be the central extension~\eqref{central extension}. It is enough to show that $\Ker(\psi) = Z(C)$. To see the non-trivial inclusion, consider any $x \in Z(C)$, then by Proposition~\ref{equivalence} we have that $\psi(x) \in Z_{\alpha}^\curlywedge(L)=0$, hence $x \in \Ker(\psi)$.
	
	Assume now that $(L, \alpha_L)$ is a capable Hom-Lie algebra, i.e., there exists a Hom-Lie algebra $(G, \alpha_G)$ such that $L \cong G/Z(G)$. Then, there is a surjective homomorphism of Hom-Lie algebras $\pi \colon (G, \alpha_G) \twoheadrightarrow (L, \alpha_L)$ such that $\Ker(\pi) = Z(G)$. Consider a diagram of free presentations as~\eqref{free present diagr}, then we have the following commutative diagram:
	\[ \xymatrix{
		0 \ar[r] & (S, \alpha_S) \ar[r] \ar@{-->}[d] & (F, \alpha_F) \ar[r]^-{\tau} \ar@{>>}[d]^{\rho} & (L, \alpha_L) \ar[r] \ar@{=}[d] & 0\\
		0 \ar[r] & Z(G) \ar[r] & (G, \alpha_G) \ar[r]^-{\pi} & (L, \alpha_L) \ar[r] & 0
	}\]
	
	Moreover $\tau([F,S]) =0$, then there exists $\widetilde{\tau} \colon \left( \frac{F}{[F,S]}, \widetilde{\alpha}_F \right) \to (G, \alpha_G)$. Let $(C, \alpha_C) = \left( \frac{F}{[F,S]}, \widetilde{\alpha}_F \right)$ and $\psi = \pi \circ \widetilde{\tau}$.

	Since $\widetilde{\tau}$ is a surjective homomorphism, then $\widetilde{\tau} \big(Z(C)\big) \subseteq Z(G) = \Ker(\pi)$, hence $\psi\big(Z(C)\big) = \pi \circ \widetilde{\tau} \big(Z(C)\big) \subseteq \pi( \Ker(\pi)) =0$, i.e., $Z \left( C \right) \subseteq \Ker(\psi)$, and consequently, $\Ker(\psi) = Z \left( C \right)$. Now applying identity~\eqref{identity} we have $Z_{\alpha}^{\curlywedge(L)} = \psi\big(Z(C)\big) = \psi\big(\Ker(\psi)\big) = 0$.
\end{proof}

\begin{definition}
	The \emph{epicentre} of a Hom-Lie algebra $(L, \alpha_L)$ is the subalgebra 
	\[
	Z_{\alpha}^{\ast}(L) = \bigcap f\big(Z(G)\big)
	\]
	for all central extension $f \colon (G, \alpha_G) \twoheadrightarrow (L, \alpha_L)$.
\end{definition}

\begin{remark}
	Note that $f\big(Z(G)\big)$ is an ideal of~$(L, \alpha_L)$, so $Z_{\alpha}^{\ast}(L)$ is also an ideal of~$(L, \alpha_L)$.
\end{remark}

\begin{lemma} \label{4.9}
	Given a free presentation $0 \to (S, \alpha_S) \to (F, \alpha_F) \overset{\tau}\to (L, \alpha_L) \to 0$, consider the central extension~\eqref{central extension}. Then,
	\[
	Z_{\alpha}^{\ast}(L) = \overline{\tau} \left(Z\bigg( \frac{F}{[F,S]}, \overline{\alpha}_F \bigg) \right)
	\]
\end{lemma}
\begin{proof}
	We only need to show that $\overline{\tau} \left(Z\left( \frac{F}{[F,S]} \right) \right) \subseteq \varphi \big(Z(H)\big)$ for any central extension $0 \to (A, \alpha_A) \to (H, \alpha_H) \overset{\varphi} \to (L, \alpha_L) \to 0$.
	
	Since $(F, \alpha_F)$ is a free Hom-Lie algebra, then there exists a homomorphism $\delta \colon (F, \alpha_F) \to (H, \alpha_H)$ such that $\varphi \circ \delta = \tau$. Moreover, $\delta(S) \subseteq A$ and $\delta([F,S]) \subseteq [A, H] =0$. Therefore, we have that $\delta$ induces a homomorphism $\overline{\delta} \colon \left( \frac{F}{[F,S]},\overline{ \alpha}_F \right) \to (H, \alpha_H)$. 
	
	Let us see that $\overline{\delta} \left( Z \left( \frac{F}{[F,S]} \right) \right) \subseteq Z(H)$. Indeed, for any $\overline{f} \in Z \left( \frac{F}{[F,S]} \right)$, we have that $[{\alpha}_F^k(\overline{f}),{f'}] \in [F,S]$, for all $f' \in F$. Since as a vector space $H$ is the direct sum of $A$ and $L$, any $h \in H$ can be written as $h = \delta(f') + a, f' \in F, a \in A$, then 
	\[
	[\overline{\delta}\big(\overline{\alpha}_F^k(\overline{f})\big), h] = \overline{\delta}([\alpha_F^k(\overline{f}), f']) + [\overline{\delta}\big(\alpha_F^k(\overline{f})\big), a] = 0
	\]
	
	Finally, $\overline{\tau} \left( Z \left( \frac{F}{[F,S]} \right) \right) = \varphi \circ \overline{\delta} \left( Z \left( \frac{F}{[F,S]} \right) \right) \subseteq \varphi\big(Z(H)\big)$.
\end{proof}

\begin{proposition}
	A Hom-Lie algebra $(L, \alpha_L)$ is capable if and only if $Z_{\alpha}^{\ast}(L)=0$.
\end{proposition}
\begin{proof}
	If $(L, \alpha_L)$ is capable, then the central extension 
	\[
	0 \to (Z(H), \alpha_{H \mid}) \to (H, \alpha_H) \overset{f} \to (L, \alpha_L) \to 0
	\] 
	implies that $Z_{\alpha}^{\ast}(L) \subseteq f(Z(H)) =0$.
	
	Conversely, if $Z_{\alpha}^{\ast}(L)=0$, by Lemma~\ref{4.9} any free presentation 
	\[
	0 \to (S, \alpha_S) \to (F, \alpha_F) \overset{\tau}\to (L, \alpha_L) \to 0
	\]
	induces that $\overline{\tau} \left( Z \left( \frac{F}{[F,S]} \right) \right) = 0$, and therefore
	\[
	\left( Z \left( \frac{F}{[F,S]} \right) \right) \subseteq \frac{S}{[F,S]} = \Ker(\overline{\tau})
	\]	
	Moreover, since~\eqref{central extension} is a central extension, then $\frac{S}{[F,S]} \subseteq Z \left( \frac{F}{[F,S]} \right)$. Thus, 
	\[
	0 \to Z \left( \frac{F}{[F,S]}, \overline{\alpha}_F \right) \to \left( \frac{F}{[F,S]}, \overline{\alpha}_F \right) \overset{\overline{\tau}} \to (L, \alpha_L) \to 0
	\]
	is a central extension as well.
\end{proof}

\begin{theorem}
	Let $(A, \alpha_A)$ be a central ideal of a Hom-Lie algebra $(L, \alpha_L)$. Then, we have that $(A, \alpha_A) \subseteq Z_{\alpha}^{\ast}(L, \alpha_L)$ if and only if the homomorphism  in sequence~\eqref{6-term central} $G \colon \left(A \otimes \frac{G}{[G,G]}, \alpha_{\otimes} \right) \to \mathcal{M}(L, \alpha_L)$ associated to the central extension $0 \to (A, \alpha_A) \to (L, \alpha_L) \overset{\pi}\to \left( \frac{L}{A}, \overline{\alpha}_L \right) \to 0$ is the zero map.
\end{theorem}
\begin{proof}
	With a a similar diagram to~\eqref{free present diagr} we have the free presentation $0 \to (S, \alpha_S) \to (F, \alpha_F) \overset{\pi \circ \rho} \to (\frac{L}{A}, \overline{\alpha}_L) \to 0$.
	
	We know from the construction of the exact sequence~\eqref{6-term central} that there is an isomorphism $\Ima(G) \cong \left( \frac{[F,S]}{[F,R]}, \overline{\alpha}_{F \mid} \right)$. Then, by Lemma~\ref{4.9} we have the following commutative diagram:
	\[ \xymatrix{
		& & Z \left( \frac{F}{[F,R]} \right) \ar[r] \ar@{^{(}->}[d]& Z_{\alpha}^{\ast}(L) \ar@{^{(}->}[d]& \\
		0 \ar[r] & \left( \frac{R}{[F,R]}, \overline{\alpha}_R \right) \ar[r] \ar[ur] & \left( \frac{F}{[F,R]}, \overline{\alpha}_F \right) \ar[r]^-{\overline{\tau}} \ar@{>>}[d]^-{\epsilon} & (L, \alpha_L) \ar[r] \ar@{>>}[d]^-{\gamma} & 0\\
		& & \bullet \ar[r]^-{\sim} & \left( \frac{L}{Z_{\alpha}^{\ast}(L)}, \overline{\alpha}_L \right) &
	}\]
	
	Then,
	\begin{align*}
		G=0 &\Leftrightarrow \frac{[F,S]}{[F,R]}=0 \Leftrightarrow \left(\frac{S}{[F,R]},\overline{ \alpha}_S \right) \subseteq Z \left( \frac{F}{[F,R]} \right) \\
		{}  & \Leftrightarrow\gamma \circ \overline{\tau} \left( \frac{R}{[F, R]} \right) = \epsilon \left( \frac{R}{[F,R]} \right) =0 \\
		{}  & \Leftrightarrow(A, \alpha_A ) = \tau(S, \alpha_S) = \overline{\tau} \left(\frac{S}{[F,S]}, \overline{\alpha}_S \right) \subseteq \Ker(\gamma) = Z_{\alpha}^{\ast}(L)
	\end{align*}
\end{proof}

\begin{corollary} \label{unicentral}
	For any Hom-Lie algebra $(L, \alpha_L)$, the following statements are equivalent:
	\begin{enumerate}
		\item[a)] Any central extension $f \colon (G, \alpha_G) \twoheadrightarrow (L, \alpha_L)$ satisfies that $f(Z(G), \alpha_{G \mid}) = (Z(L), \alpha_{L \mid})$.
		\item[b)] The map $G \colon \left( Z(L) \otimes \frac{G}{[G,G]}, \alpha_{\otimes} \right) \to \mathcal{M}(L, \alpha_L)$ in sequence~\eqref{6-term central} associated to the central extension $0 \to (Z(L), \alpha_{L \mid}) \to (L, \alpha_L) \to \left( \frac{L}{Z(L)}, \overline{\alpha}_L \right) \to 0$ is the zero map.
		\item[c)] The canonical homomorphism $\mathcal{M}(L, \alpha_L) \to \mathcal{M} \left( \frac{L}{Z(L)}, \overline{\alpha}_L \right)$ is injective.
	\end{enumerate}
\end{corollary}
\begin{proof}
	It is a direct consequence of the exactness of sequence~\eqref{6-term central} associated to the central extension $0 \to (Z(L), \alpha_{L \mid}) \to (L, \alpha_L) \to \left( \frac{L}{Z(L)}, \overline{\alpha}_L \right) \to 0$.
\end{proof}

\begin{remark}
	If $(L, \alpha_L)$ satisfies any of the equivalent statements of Corollary~\ref{unicentral}, then sequence~\eqref{6-term central} associated to the central extension $(L, \alpha_L) \to \left( \frac{L}{Z(L)}, \overline{\alpha}_L \right)$ implies that 
	\[
	\mathcal{M} (L, \alpha_L) = \Ker\left( \mathcal{M}\left( \frac{L}{Z(L)}, \overline{\alpha}_L \right) \to Z(L, \alpha_L) \right).
	\]
\end{remark}

\begin{lemma} \label{2}
	Let $0 \to (N, \alpha_N) \to (L, \alpha_L) \overset{\pi}\to (\frac{L}{N}, \overline{\alpha}_L) \to 0$ be an exact sequence of Hom-Lie algebras. If $(L, \alpha_L)$ is perfect, then
	\[
	\Ker \left( \HH_2^{\alpha}(L, \alpha_L) \to \HH_2^{\alpha} \left( \frac{L}{N}, \overline{\alpha}_L \right) \right) = \Ker \left( \theta_{N, L} \colon (N \curlywedge L, \alpha_{N \curlywedge L}) \to (N, \alpha_N) \right)
	\]
\end{lemma}
\begin{proof}
	First observe that $(\frac{L}{N}, \overline{\alpha}_L)$ is also a perfect Hom-Lie algebra. Then, consider the following diagram of exact rows:
	\[ \xymatrix{
		& (N \curlywedge L, \alpha_{N \curlywedge L}) \ar[r] \ar[d]^-{\theta_{N, L}} & (L \curlywedge L, \alpha_{L \curlywedge L}) \ar[r]^-{\pi \curlywedge \pi} \ar[d]^-{\theta_{L, L}} & (\frac{L}{N} \curlywedge \frac{L}{N}, \overline{\alpha}_{L \curlywedge L}) \ar[r] \ar[d]^-{\theta_{L/N, L/N}}& 0 \\
		0 \ar[r] & (N, \alpha_N) \ar[r] & (L, \alpha_L) \ar[r]^-{\pi} & (\frac{L}{N}, \overline{\alpha}_L) \ar[r] & 0
	} \]
	and apply the Snake lemma.
\end{proof}

\begin{remark}
	Note that if $(N, \alpha_N)$ is a central ideal in $(L, \alpha_L)$, then 
	\[
	\Ker \big( \theta_{N, L} \colon (N \curlywedge L, \alpha_{N \curlywedge L}) \to (N, \alpha_N) \big) = (N \curlywedge L, \alpha_{N \curlywedge L}).
	\]
\end{remark}

\begin{theorem}
	For any perfect Hom-Lie algebra $(L, \alpha_L)$, the following statements hold:
	\begin{enumerate}
		\item[a)] $Z_{\alpha}^{\curlywedge}(L)$ is the smallest central ideal of $(L, \alpha_L)$ containing all the central ideals~$(N, \alpha_N)$ for which the canonical morphism $\mathcal{M} (L, \alpha_L) \to \mathcal{M} \left( \frac{L}{N}, \overline{\alpha}_L \right)$ is a monomorphism (or, equivalently, for which the canonical surjective homomorphism ${\pi \curlywedge \pi} \colon (L \curlywedge L, \alpha_{L \curlywedge L}) \to (\frac{L}{N} \curlywedge \frac{L}{N}, \overline{\alpha}_{L \curlywedge L})$ is an isomorphism).
		\item[b)] $Z_{\alpha}^{\curlywedge}(L) = Z_{\alpha}^{\ast}(L)$.
	\end{enumerate}
\end{theorem}
\begin{proof}
	The first part is obtained from the following commutative diagram constructed for any central ideal $(N, \alpha_N)$ and using Lemmas~\ref{1} and~\ref{2}:
	%	For any central ideal $(N, \alpha_N)$ we can use Lemmas~\ref{1} and~\ref{2} to construct the following commutative diagram:
	\[ \xymatrix{
		(N \curlywedge L, \alpha_{N \curlywedge L}) \ar[r] \ar@{=}[d] & \HH_2^{\alpha}(L, \alpha_L) \ar[r] \ar@{>->}[d] & \HH_2^{\alpha} \left(\frac{L}{N}, \overline{\alpha}_L \right) \ar@{>->}[d] & \\
		(N \curlywedge L, \alpha_{N \curlywedge L}) \ar[r] & (L \curlywedge L, \alpha_{L \curlywedge L}) \ar[r] \ar@{>>}[d] & \left(\frac{L}{N} \curlywedge \frac{L}{N}, \overline{\alpha}_{L \curlywedge L} \right) \ar@{>>}[d] \ar[r]& 0\\
		& (L, \alpha_L) \ar@{>>}[r]^-{\pi} & \left(\frac{L}{N}, \overline{\alpha}_L \right) &
	} \]
	
	The second part is a direct consequence of Corollary~\ref{unicentral} and the previous diagram where the central ideal $(N, \alpha_N)$ is the epicentre of $(L, \alpha_L)$.
\end{proof}

\section*{Acknowledgements}
The authors would like to thank the referee for his helpful comments and suggestions that improved the manuscript.

%\bibliography{biblio_xabi}
%\bibliographystyle{amsplain-nodash}

\end{document}